\documentclass[hidelinks]{siamart171218}

\usepackage{amssymb}

\newsiamremark{remark}{Remark}
\newcommand{\Z}{\mathbb{Z}}

\title{
Extremal problems on the hypercube and 
\\ 
the codegree Tur\'{a}n density of complete {\Large{\MakeLowercase{$r$}}}-graphs
}

\author{Alexander Sidorenko\thanks{\email{sidorenko.ny@gmail.com.com}}}

\headers{Extremal problems on the hypercube}{Alexander Sidorenko}

\begin{document}

\maketitle

\begin{abstract}
Let $G$ be a finite abelian group, 
and $r$ be a multiple of its exponent. 
The generalized Erd\H{o}s--Ginzburg--Ziv constant $s_r(G)$
is the smallest integer $s$
such that every sequence of length $s$ over $G$ 
has a zero-sum subsequence of length $r$. 
We show that 
$s_{2m}(\mathbb{Z}_2^d) \leq C_m 2^{d/m} + O(1)$ 
when $d\rightarrow\infty$, 
and $s_{2m}(\mathbb{Z}_2^d) \geq 2^{d/m} + 2m-1$ when $d=km$. 
We use results on $s_r(G)$ 
to prove new bounds for the codegree Tur\'{a}n density 
of complete $r$-graphs.
\end{abstract}



\section{Introduction}
In this paper, we consider three problems: 
the Sidon problem for $\Z_2^d$ (\cref{sec:Sidon}), 
the generalized Erd\H{o}s--Ginzburg--Ziv problem 
(\cref{sec:Zero-sum}), 
and the codegree Tur\'{a}n problem for complete $r$-graphs 
(\cref{sec:Turan,sec:Turan2}). 
\Cref{sec:Sidon,sec:Zero-sum} can be read 
independently from the rest of the article. 
In the proof of \cref{th:upper_general}, 
we use the notion of $r$-graphs. 
The necessary definitions and notation are given below.

An $r$-{\it graph} is a pair $H=(V(H),E(H))$ 
where $V(H)$ is a finite set of vertices, 
and the edge set $E(H)$ is a collection of $r$-subsets of $V(H)$. 
We denote ${\mathrm v}(H)=|V(H)|$ and ${\mathrm e}(H)=|E(H)|$.
The {\it independence number} $\alpha(H)$ 
is the maximum size of a subset of $V(H)$ 
which contains no edges of $H$. 
The {\em degree} of a subset $A \subseteq V(H)$ 
is the number of edges of $H$ which contain $A$.
For $0 \leq l \leq r$, let $\Delta_l(H)$ denote
the maximum degree among of $l$-subsets of $V(H)$. 
Notice that $\Delta_0(H) = {\mathrm e}(H)$ and 
\begin{equation}\label{monotonicity}
  \frac{\Delta_0(H)}{\binom{n}{r}} \;\; \leq \;\; 
  \frac{\Delta_1(H)}{\binom{n-1}{r-1}} \;\; \leq \;\; 
    \cdots \;\; \leq \;\;
  \frac{\Delta_{r-1}(H)}{\binom{n-(r-1)}{1}} \; .
\end{equation}

\section{Codegree Tur\'{a}n density}\label{sec:Turan}

The classical Tur\'{a}n number, $T(n,k,r)$, 
is the minimum number of edges 
in an $n$-vertex $r$-graph $H$ with $\alpha(H) < k$. 
Correspondingly, $\binom{n}{r} - T(n,k,r)$ 
is the largest number of edges in an $n$-vertex $r$-graph 
that does not contain a complete subgraph on $k$ vertices.
There exists the limit 
$t(k,r) = \lim_{n\rightarrow\infty} T(n,k,r) / \binom{n}{r}$. 
The exact values of Tur\'{a}n numbers for $r=2$ were found 
by Mantel \cite{Mantel:1907} in the case $k=3$, 
and by Tur\'{a}n \cite{Turan:1941} for all $k$. 
In particular, $t(k,2) = 1/(k-1)$. 
For $k>r>2$, not a single value $t(k,r)$ is known.
For details, see surveys \cite{Keevash:2011,Sidorenko:1995}.

One of the ways to generalize the Tur\'{a}n numbers is 
\[
  T_l(n,k,r) \; = \; 
  \min \{ \Delta_l(H) : \; {\mathrm v}(H)=n, \; \alpha(H)<k \} \; .
\]
Notice that $T(n,k,r)=T_0(n,k,r)$. 
Lo and Markstr\"{o}m \cite{Lo:2014} proved the existence of the limit
$
  t_l(k,r) \; = \; 
  \lim_{n\rightarrow\infty} T_l(n,k,r) \left/ \binom{n-l}{r-l} \right. \; .
$
Inequalities~(\ref{monotonicity}) imply 
\[
  t_0(k,r) \; \leq \;
  t_1(k,r) \; \leq \;
    \ldots \; \leq \;
  t_{r-1}(k,r) \; .
\]

The case $l=1$ is known as a {\it Zarankiewicz type problem} 
(see \cite[Chapter~3]{Simonovits:1997}), 
and $t_1(k,r)=t_0(k,r)=t(k,r)$. 
The problem of determining $t_l(k,r)$ has been studied 
in \cite{Czygrinow:2001,Falgas-Ravry:2013,Lo:2014,Mubayi:2007}
in its complimentary form 
(see also Chapter~13.2 of survey \cite{Keevash:2011}). 
In notation of \cite{Lo:2014}, $t_l(k,r)=1-\pi_l\left( K_k^r\right)$.
The case $l=r-1$ was first introduced by Mubayi and Zhao \cite{Mubayi:2007} 
under the name of {\it codegree density}. 
Lo and Markstr\"{o}m \cite{Lo:2014} proved 
that for all $l=1,2,\ldots,r-1$, 
\begin{equation}\label{eq:Lo}
  t_l(k,r) \; \leq \; t_{l-1}(k-1,r-1) \; .
\end{equation}

To simplify notation for the codegree density, 
we define $\tau(k,r)=t_{r-1}(k,r)$.
The known upper bounds for $\tau(k,r)$ 
follow from \cref{eq:Lo} 
and upper bounds for the classical Tur\'{a}n density: 
$\tau(k,r) \leq 1/(k-r+1)$. 
Czygrinow and Nagle \cite{Czygrinow:2001} conjectured that $\tau(4,3)=1/2$. 
Lo and Markstr\"{o}m \cite{Lo:2014} extended this conjecture to 
$\tau(r+1,r)=1/2$. 

We will prove upper bounds on the codegree density 
which are significantly better than 
$\tau(k,r) \leq 1/(k-r+1)$. 

In \cref{sec:Sidon,sec:Zero-sum} of this article, 
we study Sidon sets and zero-sum-free sequences 
in group ${\mathbb Z}_2^d$.
The results of \cref{sec:Sidon,sec:Zero-sum}
are used in \cref{sec:Turan2} 
to obtain new upper bounds for $\tau(k,r)$ when $k-r$ is small.
In particular, for $r=3$, we prove
\begin{equation}\label{eq:r_3_small}
  \tau(2a_d + 1, 3) \; \leq \; 3^{-d} \; ,
\end{equation}
where $a_d$ is the maximum size of a cap in the affine geometry $AG(d,3)$ 
($a_2=4,\; a_3=9,\; a_4=20,\; a_5=45,\; a_6=112$). 
For $r \geq 4$, we prove
\begin{equation}\label{eq:small_r-k}
  \tau(r+2,r) \leq 1/4 ,  \;\;\;\;
  \tau(r+3,r) \leq 1/8 ,  \;\;\;\;
  \tau(r+5,r) \leq 1/16 ,  \;\;\;\;
\end{equation}
and in general,
\begin{equation}\label{eq:not_so_small_r-k}
  \tau(r+b_d,r) \; \leq \; 2^{-d} \; ,
\end{equation}
where 
$b_d = \lfloor \left(2^{d+1} - 7/4\right)^{1/2} - 1/2 \rfloor$.
Notice that $d=1$ in \cref{eq:not_so_small_r-k} 
gives $\tau(r+1,r) \leq 1/2$ 
which is in line with the conjecture of Lo and Markstr\"{o}m.

\section{Sidon problem for $\Z_2^d$}\label{sec:Sidon}

A {\it Sidon set} $A$ in an abelian group $G$ is a set with the property 
that all pairwise sums of its elements are different 
(see \cite{Babai:1985}). 
If $G$ is finite, let $\beta(G)$ denote the largest size of its Sidon set. 
Obviously, $\binom{\beta(G)}{2} \leq |G|$. 

We denote by $\Z_k^d$ 
the group of $d$-dimensional vectors over ${\mathbb Z}_k$. 

\begin{theorem}\label{th:SidonUpper}
\[
  \beta\left(\Z_2^d\right) \; \leq \; 
  \sqrt{2^{d+1} - \frac{7}{4}} + \frac{1}{2} \; .
\]
\end{theorem}

\begin{proof}
Let $A$ be a Sidon set in $\Z_2^d$. 
Since two unequal elements can not have zero sum, 
$\binom{|A|}{2} \leq 2^d - 1$ 
which results in 
$|A| \leq (2^{d+1} - 7/4)^{1/2} + 1/2$. 
\end{proof}

\begin{theorem}[\cite{Lindstrom:1969}]\label{th:SidonLower}
For even values of $d$, 
\[
  \beta\left( \Z_2^d \right) \; \geq \; 2^{d/2} \; .
\]
\end{theorem}


\begin{theorem}\label{th:SidonExact}
$  \beta\left(\Z_2^1\right) = 2$, 
$\;\beta\left(\Z_2^2\right) = 3$,
$\;\beta\left(\Z_2^3\right) = 4,\;$ and 
$\;\beta\left(\Z_2^4\right) = 6$. 
\end{theorem}

\begin{proof}
Let $A_d$ be the set of vectors from $\Z_2^d$ 
with at most one non-zero component. 
This is a Sidon set, and $|A_d|=d+1$ 
provides a lower estimate for $d \leq 3$. 
For $d=4$, $\;A_4$ with the addition of vector $(1,1,1,1)$ 
demonstrates that $\beta\left(\Z_2^4\right) \geq 6$.
The matching upper bounds follow from \cref{th:SidonUpper}.
\end{proof}

\section{Zero-sum-free sequences in $\Z_2^d$}\label{sec:Zero-sum}

Let $G$ be a finite abelian group with exponent $\exp(G)$ 
(that is the least common multiple of the orders of its elements). 
The Erd\H{o}s--Ginzburg--Ziv constant $s(G)$ is the smallest integer $s$ 
such that every sequence of length $s$ over $G$ 
has a zero-sum subsequence of length $\exp(G)$ 
(see 
\cite{
Edel:2007,
Ellenberg:2017,
Gao:2006,Gao:2003,Harborth:1973,Kemnitz:1983,Reiher:2007}). 
In 1961, Erd\H{o}s, Ginzburg, and Ziv \cite{Erdos:1961} proved 
$s(\Z_k)=2k-1$. 
Kemnitz' conjecture, $s(\Z_k^2)=4k-3$ (see~\cite{Kemnitz:1983}), 
was open for more than twenty years 
and finally was proved by Reiher~\cite{Reiher:2007} in 2007.

Harborth \cite{Harborth:1973} introduced constant $g(G)$ 
which is the smallest integer $g$ such that 
every subset of size $g$ in $G$ contains $\exp(G)$ elements with zero sum. 
When $\exp(G)=3$, the sum of three elements of $G$ is zero if and only if 
they form an arithmetic progression. 
It is known that 
\begin{equation}\label{eq:Z_3_d}
  s(\Z_3^d) \; = \; 2g(\Z_3^d) - 1 \; ,
\end{equation}
and $a_d = g(\Z_3^d) - 1$ is the maximum size of a cap in the affine geometry $AG(d,3)$ 
(see \cite{Edel:2007}). 
The known exact values 
(see \cite{Edel:2007,Potechin:2008}) are 
$a_2=4,\; a_3=9,\; a_4=20,\; a_5=45,\; a_6=112$. 
Ellenberg and Gijswijt \cite{Ellenberg:2017} proved 
$g(\Z_3^d) - 1 \leq \eta^d$, where 
$\eta = (3/8) \sqrt[3]{207+33\sqrt{33}} < 2.756$.
Consequently, 
\begin{equation}\label{eq:Z_3}
  s(\Z_3^d) \; \leq \;  2 \eta^d + 1 \; .
\end{equation}

The following generalization of the Erd\H{o}s--Ginzburg--Ziv constant
was introduced by Gao~\cite{Gao:2003}. 
If $r$ is a multiple of $\exp(G)$ 
then $s_r(G)$ denotes the smallest integer $s$
such that every sequence of length $s$ over $G$ 
has a zero-sum subsequence of length $r$. 
(Notice that if $r$ is not a multiple of $\exp(G)$ 
then there is an element $x \in G$ whose order is not a divisor of $r$,
and the infinite sequence 
$x,x,x,\ldots$ contains no zero-sum subsequence of length $r$.) 
Obviously, $s_{\exp(G)}(G)=s(G)$. 
Constants $s_r(G)$ were studied in 
\cite{Bitz:2017,Gao:2014,Gao:2006b,Gao:2003,Han:2018,He:2016,Kubertin:2005}.
In the case when $k$ is a power of a prime, 
Gao proved $s_{km}(\Z_k^d) = km + (k-1)d$ for $m \geq k^{d-1}$ 
(see \cite{Gao:2006b,Kubertin:2005}) 
and conjectured that
\begin{equation}\label{eq:GaoConj}
  s_{km}(\Z_k^d) \; = \; km + (k-1)d
    \;\;\;\;\mbox{for}\;\; km > (k-1)d \; .
\end{equation}

The Harborth constant $g(G)$ allows a similar generalization. 
We say that $A \subseteq G$ is a {\it zero-free set of rank} $r$ 
if the sum of any $r$ distinct elements of $A$ is non-zero. 
When $r$ is a multiple of $\exp(G)$, we denote 
the largest size of such set by $\beta_r(G)$. 
Obviously, $\beta_{\exp(G)}(G) = g(G) - 1$. 

In \cref{sec:Sidon}, we studied $\beta(\Z_2^d)$, 
the largest size of a Sidon set in $\Z_2^d$. 
It is easy to see that a zero-free set of rank $4$ in $\Z_2^d$ 
is the same as a Sidon set. 
Hence,
$
  \beta_4(\Z_2^d) = \beta(\Z_2^d) .
$
Note that a zero-free set of rank $2m$ in $\Z_2^d$, where $m \geq 3$, 
may contain different $m$-subsets with the same sum, 
for example, $x_1+x_2+x_3 = x_1+x_4+x_5$. 
Nevertheless, we will prove that both 
$\beta_{2m}(\Z_2^d)$ and $s_{2m}(\Z_2^d)$ are of order
$2^{d/m}$ as $d\rightarrow\infty$.

\begin{theorem}\label{th:multiset}
\[
  s_{2m}(\Z_2^d) \; \leq \; \beta_{2m}(\Z_2^d) \; + \; 2m-1 \; .
\]
\end{theorem}

\begin{proof}
Consider a sequence $S$ of length $\beta + 2m-1$ over $\Z_2^d$ 
where $\beta = \beta_{2m}(\Z_2^d)$. 
We are going to show that $S$ contains 
a zero-sum subsequence of size $2m$. 
For each $x \in \Z_2^d$, denote by $k(x)$ 
the number of appearances of $x$ in $S$. 
Let $B$ be the set of elements $x \in \Z_2^d$ 
such that $k(x) \geq 1$. 
If $|B| > \beta$, 
a zero-sum subsequence exists by the definition of 
$\beta_{2m}(\Z_2^d)$.
We may assume $|B| \leq \beta$. 
Let $k'(x)$ be the largest even number that does not exceed $k(x)$. 
Then 
\[ \sum_{x \in B} k'(x) \geq \sum_{x \in B} (k(x)-1) 
   = \sum_{x \in B} k(x) - |B| = (\beta + 2m-1) - |B| \geq 2m-1 .
\]
Since the values of $k'(x)$ are even, $\sum_{x \in B} k'(x) \geq 2m$.
Select a set of even numbers $k''(x)$ 
such that $k''(x) \leq k'(x)$ and $\sum_{x \in B} k''(x) = 2m$. 
Then $k''(x)$ appearances of every $x \in B$ in $S$ 
constitute a zero-sum subsequence of length $2m$. 
\end{proof}

From \cref{th:SidonUpper,th:multiset} we get

\begin{corollary}\label{th:s_4}
\[
  s_4(\Z_2^d) \; \leq \; \sqrt{2^{d+1} - \frac{7}{4}} + \frac{7}{2} \; .
\]
\end{corollary}

\begin{theorem}\label{th:tilde4}
\[
  s_4(\Z_2^d) \; = \; \beta(\Z_2^d) \; + \; 3 \; .
\]
\end{theorem}

\begin{proof}
Let $A=\left\{x_1,x_2,\ldots,x_\beta\right\}$ 
be a Sidon set in $\Z_2^d$ where $\beta = \beta(\Z_2^d)$. 
Notice that in the sequence $x_1,x_2,\ldots,x_\beta$ 
all subsequences of size $2$ and $4$ have non-zero sums. 
Consider the sequence 
$x_1,x_2,\ldots,x_\beta,x_{\beta + 1},x_{\beta + 2}$ 
where  $x_{\beta + 2}=x_{\beta + 1}=x_\beta$. 
All 4-element subsequences of this sequence will have non-zero sums. 
Hence, $s_4(\Z_2^d) \geq \beta(\Z_2^d) + 3$. 
The opposite inequality follows from \cref{th:multiset}.
\end{proof}

\begin{theorem}\label{th:upper_general}
For each $m$, there is a constant $C_m$ such that
\[
  \beta_{2m}(\Z_2^d) \; \leq \; C_m 2^{d/m} + O(1) 
    \;\;{\rm as}\;\; d\rightarrow\infty \; .
\]
\end{theorem}

A subset of edges in an $r$-graph is called {\it independent} 
if they are pairwise disjoint. 
In order to prove \cref{th:upper_general}, 
we need the following two lemmas.

\begin{lemma}\label{th:lemma}
If an $r$-graph $H$ has no more than $\lambda$ independent edges, then 
${\mathrm e}(H) \leq \lambda \cdot (1 + r \cdot (\Delta_1(H)-1))$.
\end{lemma}

\begin{proof}
We will use induction on $\lambda$. 
The basis for $\lambda=0$ is trivial. 
Suppose, the statement of the lemma holds for $\lambda < k$. 
We will show that it holds for $\lambda = k$ as well. 
Select an arbitrary edge $A$ in $H$ and remove $r$ vertices that form $A$ 
together with all edges that intersect $A$. 
The resulting $r$-graph $H_1$ has no more than $k-1$ independent edges, 
hence 
${\mathrm e}(H_1) \leq (k-1)(1 + r \cdot (\Delta_1(H_1)-1))$. 
The number of edges we have removed is at most 
$1 + r \cdot (\Delta_1(H)-1)$, 
hence 
${\mathrm e}(H) \leq {\mathrm e}(H_1) + 1 + r \cdot (\Delta_1(H)-1)
  \leq k \cdot (1 + r \cdot (\Delta_1(H)-1))$.
\end{proof}

\begin{lemma}[The Erd\H{o}s--Ko--Rado theorem \cite{Erdos_Ko_Rado}]\label{th:EKR}
Let $H$ be an $r$-graph with $n \geq 2r$ vertices. 
If every pair of edges in $H$ has non-empty intersection, 
then ${\mathrm e}(H) \leq \binom{n-1}{r-1}$.
\end{lemma}

\begin{proof}[Proof of \cref{th:upper_general}]
For a subset $X \subset \Z_2^d$, 
let $\Sigma(X)$ denote the sum of its elements. 
Let $n=\beta_{2m}(\Z_2^d)$, 
and $A \subset \Z_2^d$ be a zero-free set of rank $2m$ and size $n$. 
For each $r=2,3,\ldots,m$, let $q(r)$ denote 
the integer $q\in\{0,1,\ldots,r-1\}$ such that 
$m+q \equiv 0 \pmod{r}$. 
Denote $\lambda_r = 2(m+q(r))/r + 2r-q(r)-3$ if $q(r)>0$, 
and $\lambda_r = 2m/r - 1$ if $q(r)=0$. 
It is easy to see that $\lambda_r$ is a positive integer. 
We say that an $r$-subset $X \subseteq A$ is {\it exceptional} 
if $q(r)>0$ 
and there exist $r$-subsets $X_1,X_2,\ldots,X_{\lambda_r} \subseteq A$ 
such that $X_1,X_2,\ldots,X_{\lambda_r},X$ are pairwise disjoint 
and $\Sigma(X_1)=\Sigma(X_2)=\ldots=\Sigma(X_{\lambda_r})=\Sigma(X)$. 

Our first step will be to prove that if $q(r)>0$ 
then two exceptional $r$-subsets can not have intersection of size $q(r)$. 
Indeed, let $X$ and $Y$ be exceptional $r$-subsets and 
$|X \cap Y| = q(r)$. 
There exist $r$-subsets $X_1,X_2,\ldots,X_\lambda$ 
and $Y_1,Y_2,\ldots,Y_\lambda$ such that 
$X_1,X_2,\ldots,X_\lambda,X$ are pairwise disjoint, 
$Y_1,Y_2,\ldots,Y_\lambda,Y$ are pairwise disjoint, 
$\Sigma(X_1)=\Sigma(X_2)=\ldots=\Sigma(X_\lambda)=\Sigma(X)$, and 
$\Sigma(Y_1)=\Sigma(Y_2)=\ldots=\Sigma(Y_\lambda)=\Sigma(Y)$, 
where $\lambda=\lambda_r$. 
It is possible that $\Sigma(X)=\Sigma(Y)$ 
and $X_i=Y_j$ for some $i,j$. 
Notice that $X-Y$ can intersect at most $r-q(r)$ subsets among 
$Y_1,Y_2,\ldots,Y_\lambda$. 
As $\lambda > r-q(r)$, there is an index $j$ such that 
$X \cap Y_j = \emptyset$. 
Similarly, $Y-X$ can intersect at most $r-q(r)$ subsets among 
$X_1,X_2,\ldots,X_\lambda$. 
Also, $Y_j$ can intersect at most $r$ subsets among 
$X_1,X_2,\ldots,X_\lambda$. 
Since $\lambda - (r-q(r)) - r = 2k-3$ with $k=(m+q(r))/r$, 
there exist $2k-3$ indices 
$1 \leq i_1 < i_2 < \ldots < i_{2k-3} \leq \lambda$ such that 
$\left(X_{i_1} \cup X_{i_2} \cup \ldots \cup X_{i_{2k-3}}\right)
 \cap (Y \cup Y_j) = \emptyset$. 
Among $2k$ subsets $X,X_{i_1},X_{i_2},\ldots,X_{i_{2k-3}},Y,Y_j$, 
the only pair with non-empty intersection is $\{X,Y\}$. 
Let 
\[
  B \; = \; (X \cup X_{i_1} \cup X_{i_2} \cup \ldots \cup X_{i_{2k-3}}
  \cup Y \cup Y_j) \; - \; (X \cap Y) \; .
\]
Then $|B| = 2kr - 2 |X \cap Y| = 2kr - 2q(r) = 2m$ 
and $\Sigma(B) = (2k-2)\Sigma(X) + 2\Sigma(Y) - 2\Sigma(X \cap Y) = 0$ 
which contradicts the assumption 
that $A$ is a zero-free set of rank $2m$. 

Our second step is to obtain an upper bound 
on the number of exceptional $r$-subsets. 
Fix $B \subset A$ where $|B| = q(r) > 0$ 
and consider a family $\mathcal{F}_B$ of subsets $F \subset A-B$ 
such that $|F|=r-q(r)$ and $F \cup B$ is an exceptional $r$-subset. 
Then any two members of $\mathcal{F}_B$ must have non-empty intersection. 
Since $n = \beta_{2m}(\Z_2^d) \geq 2m-1$ and $r < m$, 
we have $|A-B| = n-q(r) \geq 2(r-q(r))$. 
By \cref{th:EKR}, 
$|\mathcal{F}_B| \leq \binom{|A-B|-1}{r-q(r)-1} = \binom{n-q(r)-1}{r-q(r)-1}$. 
Then the total number of exceptional $r$-subsets is at most
\[
  \binom{n}{q(r)} \binom{n-q(r)-1}{r-q(r)-1}
    \; = \;
  \frac{r-q(r)}{n-q(r)} \binom{n}{r} \binom{r}{q(r)} \; ,
\]
which is a polynomial in $n$ of degree $r-1$.

In the case $q(r)>0$, let $G_r$ denote an $m$-graph with vertex-set $A$ 
where an $m$-subset $B \subseteq A$ is an edge 
if $B$ contains an exceptional $r$-subset. 
Then 
$
  {\mathrm e}(G_r)
    \leq 
  \frac{r-q(r)}{n-q(r)} \binom{n}{r} \binom{r}{q(r)} \binom{n-r}{m-r}
$. 
Denote 
\[
  P_m(n) \; = \; 
  \sum_{\stackrel{\scriptstyle r=2}{q(r)>0}}^{m-1} 
    \frac{r-q(r)}{n-q(r)} \binom{n}{r} \binom{n-r}{m-r} \binom{r}{q(r)}
  \; .
\]
Notice that $P_m(n)$ 
is a polynomial in $n$ of degree at most $m-1$. 
For $r=1,2,\ldots,m$ and $z\in\Z_2^d$, 
we denote by $H_r(z)$ an $r$-graph with vertex set $A$ 
whose edges are $r$-subsets $X$ such that 
$\Sigma(X)=z$ and 
$X$ does not contain an exceptional subset.
Notice that 
\[
  \sum_{z\in\Z_2^d} {\mathrm e}(H_m(z))
  \;\; \geq  \;\;
  \binom{n}{m} - 
    \sum_{\stackrel{\scriptstyle r=2}{q(r)>0}}^{m-1} {\mathrm e}(G_r)
  \;\; \geq \;\;
  \binom{n}{m} \; - \; P_m(n) \; .
\]

As the third step, we will obtain an upper bound on ${\mathrm e}(H_r(z))$. 
Let $N_1=1$ and 
$N_r = \lambda_r \cdot (1 + r \cdot (N_{r-1} - 1))$ 
for $r=2,3,\ldots,m$. 
We are going to prove ${\mathrm e}(H_r(z)) \leq N_r$ 
for every $r \leq m$ and every $z \in\Z_2^d$. 
We will use induction on $r$. 
The case $r=1$ serves as the induction base.  
Indeed, $H_1(z)$ has either $1$ edge 
(that is $z$ itself) if $z \in A$, 
or no edges if $z \notin A$. 
Now we will prove the induction step from $r-1$ to $r$. 
Notice that the degree of vertex $x$ in $H_r(z)$ 
is at most ${\mathrm e}(H_{r-1}(z+x)) \leq N_{r-1}$. 
If $q(r)>0$ 
then  $H_r(z)$ has no exceptional $r$-subset as its edge, 
hence, it has at most $\lambda_r$ independent edges.
If  $q(r)=0$, then $r$ is a divisor of $m$, 
and $H_r(z)$ can not have $2m/r = \lambda_r + 1$ independent edges: 
their union would be an $(2m)$-subset with zero sum. 
We apply \cref{th:lemma} to $H_r(z)$ with 
$\lambda = \lambda_r$ and $\Delta_1(H_r(z)) \leq N_{r-1}$,
to get 
$
  {\mathrm e}(H_r(z)) \leq
  \lambda_r \cdot (1 + r \cdot (N_{r-1} - 1)) = N_r \; .
$ 

We recall that 
$\sum_{z\in\Z_2^d} {\mathrm e}(H_m(z)) \geq \binom{n}{m} - P_m(n)$,
where $P_m(n)$ is a polynomial of order less than $m$. 
On the other hand, ${\mathrm e}(H_m(z)) \leq N_m$ 
for every $z\in\Z_2^d$. 
Therefore, 
$\binom{n}{m} - P_m(n) \leq 2^d N_m$. 
Since $n=\beta_{2m}(\Z_2^d)$, we get
$\beta_{2m}(\Z_2^d) \leq \left(2^d m! N_m\right)^{1/m} + O(1)$ 
as $d\rightarrow\infty$. 
\end{proof}

\begin{remark}
In the proof of \cref{th:upper_general}, one may estimate 
$\lambda_r < 2\left(\frac{m}{r}+r\right)$, and hence, 
$\left(C_m\right)^m < m! \prod_{r=2}^m r\lambda_r < m! \prod_{r=2}^m 2(m+r^2)$. 
This implies $C_m = O(m^3)$ as $m\rightarrow\infty$. 
For small $m$, 
$C_3=60^{1/3} < 3.9149$ and $C_4=3288^{1/4} < 7.5724$.
\end{remark}

The next result is a generalization of \cref{th:SidonLower}. 

\begin{theorem}\label{th:construction}
If $d$ is a multiple of $m$ then 
\[
  \beta_{2m}(\Z_2^d) \; \geq \; 2^{d/m} \; ,
                     \;\;\;\;\;\;\;\;
      s_{2m}(\Z_2^d) \; \geq \; 2^{d/m} + 2m - 1 \; .
\]
\end{theorem}

\begin{proof}
Let $d=m \cdot k$. 
Since ${\mathbb Z}_2^k$ is the additive group of 
$\mathbb{GF}\left(2^k\right)$, 
the elements of $\Z_2^{mk}$ 
can be represented by vectors $(x_1,x_2,\ldots,x_m)$ 
where $x_i \in \mathbb{GF}\left(2^k\right)$. 
Let $A$ be a set of $2^k$ vectors $(x,x^3,x^5,\ldots,x^{2m-1})$ 
where $x \in \mathbb{GF}\left(2^k\right)$. 
We are going to prove that $A$ is a zero-free set of rank $2n$ 
for each $n=1,2,\ldots,m$. 
Indeed, suppose that 
$x_1,x_2,\ldots,x_{2n} \in \mathbb{GF}\left(2^k\right)$ 
and $\sum_{i=1}^{2n} (x_i)^r = 0$ for every odd $r \leq 2n-1$. 
We need to show that there are $i,j$ such that $x_i=x_j$, $\;i \neq j$.
As 
$\left(\sum_{i=1}^{2n} (x_i)^r\right)^2 = 
 \sum_{i=1}^{2n} (x_i)^{2r}$, 
we get $\sum_{i=1}^{2n} (x_i)^r = 0$ for all $r \leq 2n-1$. 
Let $M=\left[M_{ij}\right]$ be a square matrix of order $2n$ 
over $\mathbb{GF}\left(2^k\right)$, 
where $M_{ij}=(x_i)^{j-1}$. 
Notice that 
$(1,1,\ldots,1) \cdot M = (0,0,\ldots,0)$, 
so $\det (M) = 0$ 
(where $0$ and $1$ are elements of $\mathbb{GF}\left(2^k\right)$). 
On the other hand, 
$\det (M) = \prod_{1 \leq i<j \leq 2n} (x_i-x_j)$ 
which means that there are $i,j$ such that $x_i=x_j$, $\;i \neq j$. 
We have proved by now that $\beta_{2m}(\Z_2^{mk}) \geq 2^k$. 

To prove the lower bound for $s_{2m}(\Z_2^{mk})$, 
select an element $a \in A$ and consider a sequence $S$ of length $2^k+2m-2$ 
where $a$ appears $2m-1$ times and each other element from $A$ appears once. 
We claim that $S$ does not contain a zero-sum subsequence of length $2m$. 
Indeed, suppose that such a subsequence $S'$ exists, 
and let $t$ be the number of appearances of $a$ in it. 
Let $2s$ be the largest even number that does not exceed $t$. 
Let $S''$ be obtained from $S$ by removing $2s$ copies of $a$. 
Then $S''$ is a zero-sum subsequence of length $2m-2s$ 
which does not contain multiple copies of the same element. 
It contradicts with the fact that $A$ is a zero-free set of rank $2(m-s)$. 
Therefore, $s_{2m}(\Z_2^{mk}) > 2^k+2m-2$.
\end{proof}


\section{Bounds for codegree Tur\'{a}n densities}
\label{sec:Turan2}

Let $G$ be a finite abelian group, and $r$ be a multiple of its exponent.
In \cref{sec:Zero-sum}, we defined $s_r(G)$ 
as the smallest integer $s$ such that every sequence of length $s$ 
over $G$ contains a zero-sum subsequence of length $r$.

\begin{theorem}\label{th:main}
If $G$ is a finite abelian group 
and $r$ is a multiple of $\exp(G)$ then
\[
  \tau\big( s_r(G),\; r \big) 
    \; \leq \; \frac{1}{|G|} \; .
\]
\end{theorem}

\begin{proof}
Let $H_n$ be an $r$-graph with $n$ vertices 
that are divided into $|G|$ baskets of almost equal sizes, 
each basket is associated with an element of $G$, 
and $r$ vertices form an edge 
when the sum of their associated elements is zero. 
The degrees of all $(r-1)$-subsets of $V(H_n)$ 
are $|G|^{-1} n + O(1)$ as $n\rightarrow\infty$. 
By the definition of $s_r(G)$, 
any subset of vertices of size $s_r(G)$ 
contains an edge of $H_n$.
\end{proof}

\Cref{th:main} provides the strongest results when 
$r$ is small and $|G|$ is large. 
The best cases are $G=\Z_3^d$ with $r=3$, 
and $G=\Z_2^d$ with even values of $r$. 
When $G=\Z_2^d$ and $r=4$, 
\Cref{th:main,th:s_4} 
yield 
\begin{equation}\label{eq:r_is_4}
  \tau\left(\left\lfloor\sqrt{2^{d+1} - \frac{7}{4}} 
                   + \frac{7}{2}\right\rfloor, 
      4\right)
  \; \leq \; 2^{-d} \; ,
\end{equation}
as well as 
\[
  \tau(k,4) \; \leq \; 2k^{-2}+O\left( k^{-3} \right)
    \;\;\; {\rm as} \;\; k\rightarrow\infty \; .
\]
By combining \cref{eq:Lo,eq:r_is_4}, 
we obtain \cref{eq:small_r-k,eq:not_so_small_r-k}. 
\Cref{th:multiset,th:upper_general,th:main}, 
together with \cref{eq:Lo}, yield for $r \geq 4$
\[
  \tau(k,r) \; \leq \; O\left( k^{-\lfloor r/2 \rfloor} \right)
  \;\;\;\; {\rm as} \;\; k\rightarrow\infty \; .
\]

As $s_3(\Z_3^d) = s(\Z_3^d)$, 
\Cref{th:main} together with \cref{eq:Z_3_d} 
yield \cref{eq:r_3_small}. 
\Cref{th:main} together with \cref{eq:Z_3}
yield $\tau(k,3) \leq O\left(k^{-\ln(3) / \ln(\eta)}\right)$ 
where $\eta = (3/8) \sqrt[3]{207+33\sqrt{33}}$.
As $\ln(3) / \ln(\eta) > 1.084$, it results in  
\[
  \tau(k,3) \; = \; o\left(k^{-1.084}\right)
    \;\;\;\; {\rm as} \;\; k\rightarrow\infty \; .
\]

Recently, Lo and Zhao \cite{Lo:2018} 
proved that for each $r \geq 3$,
\begin{equation}\label{eq:lo_zhao}
  c_1 \frac{\ln k}{k^{r-1}}
    \; \leq \; 
  \tau(k,r)
    \; \leq \; 
  c_2 \frac{\ln k}{k^{r-1}} 
    \;\;\;\; {\rm as} \;\; k\rightarrow\infty \; .
\end{equation}
The upper estimate in \cref{eq:lo_zhao} 
is better than our asymptotic bounds.
Nevertheless, in the case when $k-r$ is small, 
our bounds \cref{eq:r_3_small,eq:small_r-k} are still better. 

Very recently, 
Gao's conjecture~\cref{eq:GaoConj} 
was proved in~\cite{Sidorenko:2018} 
for $k=2$. 
As a consequence, we may derive from \cref{th:main} that
\[
  \tau(r+d,r) \; \leq \; 2^{-d}
    \;\;\;\;\mbox{for}\;\; r \geq 2 \lceil d/2 \rceil \; .
\]

\section*{Acknowledgments}

The author would like to thank the referees for their suggestions and comments.

\bibliographystyle{siamplain}
\bibliography{hypercube}

\end{document}